\documentclass[11pt,a4paper]{article}
\usepackage{a4wide}

\usepackage[T1]{fontenc}
\usepackage{lmodern}
\usepackage{amsthm,amsmath,amssymb}
\usepackage{hyperref}
\usepackage{microtype}
\usepackage{hyperref}
\usepackage{graphicx}
\usepackage{enumitem}

\newtheorem{theorem}{Theorem}
\newtheorem{lemma}[theorem]{Lemma}
\newtheorem{corollary}[theorem]{Corollary}
\newtheorem{proposition}[theorem]{Proposition}

\newtheorem{problem}[theorem]{Problem}

\def\inst#1{$^{#1}$}

\begin{document}

\title{On ordered Ramsey numbers of tripartite 3-uniform hypergraphs \footnote{An extended abstract version of this paper appeared in the Proceedings of Eurocomb 2021, see~\cite{balko2021ordered}}}

\author{Martin Balko\inst{1} \thanks{The first author was supported by the grant no.~18-13685Y of the Czech Science Foundation (GA\v{C}R) and by the Center for Foundations of Modern Computer Science (Charles University project UNCE/SCI/004).
This article is part of a project that has received funding from the European Research Council (ERC) under the European Union's Horizon 2020 research and innovation programme (grant agreement No 810115).} 
\and
M\'{a}t\'{e} Vizer\inst{2} \inst{3} \thanks{The second author was supported by the Hungarian National Research, Development and Innovation Office -- NKFIH under the grant SNN 129364, KH 130371 and FK 132060, by the J\'anos Bolyai Research Fellowship of the Hungarian Academy of Sciences and by the New National Excellence Program under the grant number \'UNKP-21-5-BME-361.} 
}

\maketitle

\begin{center}
{\footnotesize
\inst{1} 
Department of Applied Mathematics, \\
Faculty of Mathematics and Physics, Charles University, Czech Republic \\
\texttt{balko@kam.mff.cuni.cz}
\\\ \\
\inst{2} 
Alfr\'{e}d R\'{e}nyi Institute of Mathematics,\\ ELKH, Budapest, Hungary\\
\inst{3}
Department of Computer Science and Information Theory, \\ Budapest University of Technology and Economic \\
\texttt{vizermate@gmail.com}
}
\end{center}

\begin{abstract}
For an integer $k \geq 2$, an \emph{ordered $k$-uniform hypergraph} $\mathcal{H}=(H,<)$ is a $k$-uniform hypergraph $H$ together with a fixed linear ordering $<$ of its vertex set.
The \emph{ordered Ramsey number} $\overline{R}(\mathcal{H},\mathcal{G})$ of two ordered $k$-uniform hypergraphs $\mathcal{H}$ and $\mathcal{G}$ is the smallest $N \in \mathbb{N}$ such that every red-blue coloring of the hyperedges of the ordered complete $k$-uniform hypergraph $\mathcal{K}^{(k)}_N$ on $N$ vertices contains a blue copy of $\mathcal{H}$ or a red copy of $\mathcal{G}$.

The ordered Ramsey numbers are quite extensively studied for ordered graphs, but little is known about ordered hypergraphs of higher uniformity.
We provide some of the first nontrivial estimates on ordered Ramsey numbers of ordered 3-uniform hypergraphs.
In particular, we prove that for all $d,n \in \mathbb{N}$ and for every ordered $3$-uniform hypergraph $\mathcal{H}$ on $n$ vertices with maximum degree $d$ and with interval chromatic number $3$ there is an $\varepsilon=\varepsilon(d)>0$ such that $$\overline{R}(\mathcal{H},\mathcal{H}) \leq 2^{O(n^{2-\varepsilon})}.$$ 
In fact, we prove this upper bound for the number $\overline{R}(\mathcal{G},\mathcal{K}^{(3)}_3(n))$, where $\mathcal{G}$ is an ordered 3-uniform hypergraph with $n$ vertices and maximum degree $d$ and $\mathcal{K}^{(3)}_3(n)$ is the ordered complete tripartite hypergraph with consecutive color classes of size $n$.
We show that this bound is not far from the truth by proving $\overline{R}(\mathcal{H},\mathcal{K}^{(3)}_3(n)) \geq 2^{\Omega(n\log{n})}$ for some fixed ordered $3$-uniform hypergraph $\mathcal{H}$.
\end{abstract}

\section{Introduction}

For an integer $k \geq 2$ and a $k$-uniform hypergraph $H$, the \emph{Ramsey number} $R(H)$ is the minimum $N \in \mathbb{N}$ such that every 2-coloring of the hyperedges of the complete $k$-uniform hypergraph $K^{(k)}_N$ on~$N$ vertices contains a monochromatic subhypergraph isomorphic to $H$.
Estimating Ramsey numbers is a notoriously difficult problem.
Despite many efforts in the last 70 years, no tight bounds are known even for the complete graph $K_n$ on $n$ vertices.
Apart from some smaller term improvements, essentially the best known bounds are $2^{n/2} \leq R(K_n) \leq 2^{2n}$ by Erd\H{o}s~\cite{erdosSzekeres35} and by Erd\H{o}s and Szekeres~\cite{erdos47}.
The Ramsey numbers $R(K^{(k)}_n)$ are even less understood for $k \geq 3$.
For example, it is only known that 
\begin{equation}
\label{eq-hypergraphRamsey}
2^{\Omega(n^2)} \leq R(K^{(3)}_n) \leq 2^{2^{O(n)}},
\end{equation}
as shown by Erd\H{o}s, Hajnal, and Rado~\cite{erdosRado65}.
A famous conjecture of Erd\H{o}s, for whose proof Erd\H{o}s offered \$500 reward, states that there is a constant $c>0$ such that $R(K^{(3)}_n) \geq 2^{2^{cn}}$.

Recently, a variant of Ramsey numbers for hypergraphs with a fixed order on their vertex sets has been introduced~\cite{bckk13,clfs17}.
For an integer $k \geq 2$, an \emph{ordered $k$-uniform hypergraph} $\mathcal{H}$ is a pair $(H,<)$ consisting of a $k$-uniform hypergraph $H$ and a linear ordering $<$ of its vertex set.
An ordered $k$-uniform hypergraph $\mathcal{H}_1=(H_1,<_1)$ is an \emph{ordered subhypergraph} of another ordered $k$-uniform hypergraph $\mathcal{H}_2=(H_2,<_2)$, written $\mathcal{H}_1 \subseteq \mathcal{H}_2$, if $H_1$ is a subhypergraph of $H_2$ and $<_1$ is a suborder of $<_2$.
Two ordered hypergraphs $\mathcal{H}_1$ and $\mathcal{H}_2$ are \emph{isomorphic} if there is an isomorphism between their underlying hypergraphs that preserves the vertex orderings of $\mathcal{H}_1$ and $\mathcal{H}_2$.
Note that, up to isomorphism, there is a unique ordered complete $k$-uniform hypergraph $\mathcal{K}^{(k)}_n$ on $n$ vertices.

The \emph{ordered Ramsey number} $\overline{R}(\mathcal{H},\mathcal{G})$ of two ordered $k$-uniform hypergraphs $\mathcal{H}$ and $\mathcal{G}$ is the smallest $N \in \mathbb{N}$ such that every coloring of the hyperedges of $\mathcal{K}^{(k)}_N$ by colors red and blue contains a blue ordered subhypergraph isomorphic to $\mathcal{H}$ or a red ordered subhypergraph isomorphic to $\mathcal{G}$.
In the \emph{diagonal case} $\mathcal{H} = \mathcal{G}$, we just write $\overline{R}(\mathcal{G})$ instead of $\overline{R}(\mathcal{G},\mathcal{G})$.

The ordered Ramsey numbers are known to be finite and it is easy to see that they grow at least as fast as the standard Ramsey numbers.
Studying ordered Ramsey numbers has attracted a lot of attention lately (see the survey by Conlon, Fox, and Sudakov~\cite{cfsSurvey}), as there are various motivations coming from the field of discrete geometry.
It is known that ordered Ramsey numbers can behave quite differently than the standard Ramsey numbers, especially for sparse ordered graphs~\cite{bckk13,bjv16,clfs17}.
However, so far, the ordered Ramsey numbers have been studied mostly for ordered graphs only and very little is known about ordered Ramsey numbers of ordered $k$-uniform hypergraphs with $k \geq 3$.

In this paper, we focus on $3$-uniform hypergraphs and we prove some new bounds on the ordered Ramsey numbers of ordered tripartite $3$-uniform hypergraphs.
We also pose several new open problems in Section~\ref{sec-openProblems}.

\subsection{Preliminaries}

For an ordered $k$-uniform hypergraph $\mathcal{H}=(H,<)$ and two subsets $U$ and $V$ of vertices of~$\mathcal{H}$, we say that $U$ and $V$ are \emph{consecutive} if all vertices from $U$ precede all vertices of $V$ in~$<$.
An \emph{interval} in $\mathcal{H}$ is a subset $I$ of vertices of $\mathcal{H}$ such that for all vertices $u,v,w$ of $\mathcal{H}$ with $u<v<w$ and $u,w \in I$ we have $v \in I$.

For integers $k \geq 2$ and $\chi \geq k$, we use $K^{(k)}_\chi(n)$ to denote the \emph{complete $k$-uniform $\chi$-partite hypergraph}, that is, the vertex set of $K^{(k)}_\chi(n)$ is partitioned into $\chi$ sets of size $n$ and every $k$-tuple with at most one vertex in each of these parts forms a hyperedge.
The ordering of~$K^{(k)}_\chi(n)$, in which the color classes form consecutive intervals, is denoted by $\mathcal{K}^{(k)}_\chi(n)$.
We sometimes use $\mathcal{K}_{n,n}$ to denote $\mathcal{K}^{(2)}_2(n)$.

The \emph{degree} of a vertex $v$ in a hypergraph $H$ is the number of hyperedges of $H$ that contain~$v$.
For $d \in \mathbb{N}$, a $k$-uniform hypergraph $H$ is \emph{$d$-degenerate} if there is an ordering $v_1 \prec \cdots \prec v_t$ of vertices of $H$ such that each $v_i$ is contained in at most $d$ hyperedges of $H$ that contain a vertex from $v_1,\dots,v_{i-1}$.

For a positive integer $n$, we use $[n]$ to denote the set $\{1,\dots,n\}$.
We omit floor and ceiling sign whenever they are not crucial and we use $\log$ and $\ln$ to denote base $2$ logarithm and the natural logarithm, respectively.

\subsection{Previous results}
\label{subsec-previousResults}

The ordered Ramsey numbers of $k$-uniform ordered hypergraphs with $k \geq 3$ remain quite unexplored.
Only the ordered Ramsey numbers of so-called monotone hyperpaths are well understood due to their close connections to the famous Erd\H{o}s--Szekeres Theorem~\cite{erdosSzekeres35}; see~\cite{balko19,fpss12,moshShap14}.
A \emph{monotone hyperpath} $\mathcal{P}^{(k)}_n$ on $n$ vertices is an ordered $k$-uniform hypergraph where the hyperedges are formed by $k$-tuples of consecutive vertices.
Note that the maximum degree of a $k$-uniform monotone hyperpath is at most $k$.
Moshkovitz and Shapira~\cite{moshShap14} showed that $\overline{R}(\mathcal{P}^{(k)}_n) = {\rm tow}_{k-1}((2-o(1))n)$ for $k \geq 3$, where ${\rm tow}_h$ is the \emph{tower function} of height $h$ defined as ${\rm tow}_1(x)=x$ and ${\rm tow}_h(x) = 2^{{\rm tow}_{h-1}(x)}$ for $h \geq 2$.

Thus even for $3$-uniform hypergraphs $\mathcal{H}$ with bounded maximum degree the numbers $\overline{R}(\mathcal{H})$ can grow very fast.
We get an exponential lower bound on $\overline{R}(\mathcal{H})$ even for $3$-uniform ordered hypergraphs $\mathcal{H}$ with maximum degree $3$.
A similar result is known for ordered graphs, as for arbitrarily large values of $n$ there are ordered graphs $\mathcal{M}_n$ with $n$ vertices and maximum degree $1$ such that $\overline{R}(\mathcal{M}_n) \geq n^{\Omega(\log{n}/\log{\log{n}})}$~\cite{bckk13,clfs17}.
This superpolynomial growth rate is in sharp contrast with  the situation for unordered hypergraphs, where the Ramsey number $R(H)$ of every $k$-uniform hypergraph $H$ with bounded $k$ and with bounded maximum degree is at most linear in the number of vertices of~$H$~\cite{crst83,cfs09,cnko08,cnko09,ishingami07,nsrs08}.

Therefore, in order to obtain smaller upper bounds on the ordered Ramsey numbers, it is necessary to bound other parameter besides the maximum degree.
A natural choice is so-called interval chromatic number, which can be understood as an analogue of the chromatic number for ordered graphs due to a variant of the Erd\H{o}s--Stone--Simonovits theorem for ordered graphs proved by Pach and Tardos~\cite{pachTardos06}.
The \emph{interval chromatic number} 
$\chi_<(\mathcal{H})$ of an ordered $k$-uniform hypergraph $\mathcal{H}$ is the minimum number of intervals the vertex set of $\mathcal{H}$ can be partitioned into so that each hyperedge of $\mathcal{H}$ has at most one vertex in each of the intervals.

For ordered graphs, bounding both parameters indeed helps, as the ordered Ramsey number $\overline{R}(\mathcal{G})$ of every ordered graph $\mathcal{G}$ with bounded maximum degree $d$ and bounded interval chromatic number $\chi$ is at most polynomial in the number of vertices~\cite{bckk13,clfs17}.
Since $\mathcal{G} \subseteq \mathcal{K}^{(2)}_\chi (n)$, this result follows from the following stronger estimate proved by Conlon, Fox, Lee, and Sudakov~\cite{clfs17}: for all $d, \chi \in \mathbb{N}$, every $d$-degenerate ordered graph $\mathcal{G}$ on $n$ vertices with interval chromatic number $\chi$ satisfies
\begin{equation}
\label{eq-degreeIntChrNumber}   
\overline{R}(\mathcal{G},\mathcal{K}^{(2)}_\chi(n)) \leq n^{32 d \log{\chi}}.
\end{equation}

A natural question is whether we can also get some good upper bounds on ordered Ramsey numbers of similarly restricted classes of ordered hypergraphs.
If the interval chromatic number is bounded, then we can use a result of Conlon, Fox, and Sudakov~\cite{cfs11}, who showed that, for all positive integers $\chi \geq 3$ and $n$, \[R(K^{(3)}_\chi(n)) \leq 2^{2^{2R}n^2},\]
where $R = R(K_{\chi-1})$.
Since every ordering of $K^{(3)}_\chi(\chi n)$ contains an ordered subhypergraph isomorphic to $\mathcal{K}^{(3)}_\chi(n)$ and every ordered $3$-uniform hypergraph on $n$ vertices with interval chromatic number $\chi$ is an ordered subhypergraph of $\mathcal{K}^{(3)}_\chi(n)$, we obtain the following bound.

\begin{corollary}
\label{cor-3UnifBoundedIntChr}
For all positive integers $\chi \geq 3$ and $n$, every ordered $3$-uniform hypergraph $\mathcal{H}$ on $n$ vertices with interval chromatic number $\chi$ satisfies
\[\overline{R}(\mathcal{H}) \leq 2^{2^{2R}\chi^2n^2},\]
where $R = R(K_{\chi-1})$.
In particular, if the interval chromatic number $\chi$ of $\mathcal{H}$ is fixed, we have 
\[
\overline{R}(\mathcal{H}) \leq 2^{O(n^2)}.
\]
\end{corollary}

Note that the last bound is asymptotically tight for dense ordered hypergraphs, as a standard probabilistic argument shows that $\overline{R}(\mathcal{H}) \geq 2^{\Omega(n^2)}$ for every ordered $3$-uniform hypergraph $\mathcal{H}$ on $n$ vertices with $\Omega(n^3)$ hyperedges.
In particular, we get $\overline{R}(\mathcal{K}^{(3)}_3(n)) \geq 2^{\Omega(n^2)}$.

\section{Our results}

Since the bounds on the ordered Ramsey numbers from Corollary~\ref{cor-3UnifBoundedIntChr} are asymptotically tight for dense ordered hypergraphs with bounded interval chromatic number, we consider the sparse case with bounded maximum degree and interval chromatic number. 
The situation for ordered hypergraphs seems to be more difficult than for ordered graphs, so we focus on the first nontrivial case, which is for ordered $3$-uniform hypergraphs with interval chromatic number $3$.

Assuming the maximum degree of an ordered hypergraph $\mathcal{H}$ with $\chi_<(\mathcal{H})=3$ is sufficiently small, we obtain a better upper bound on $\overline{R}(\mathcal{H})$ than the estimate $2^{O(n^2)}$ we would get from Corollary~\ref{cor-3UnifBoundedIntChr}.
We can prove an estimate with  a subquadratic exponent even in the more general setting $\overline{R}(\mathcal{H},\mathcal{K}^{(3)}_3(n))$, where, additionally, the interval chromatic number of $\mathcal{H}$ is arbitrary.

\begin{theorem}
\label{thm-3UnifMaxDegIntChr}
Let $\mathcal{H}$ be an ordered 3-uniform hypergraph on $t$ vertices with maximum degree $d$ and let $s$ be a positive integer.
Then there are constants $C=C(d)$ and $c>0$ such that
\[\overline{R}(\mathcal{H},\mathcal{K}^{(3)}_3(s)) \leq t \cdot 2^{C(s^{2-1/(1+cd^2)})}.\]
In particular, for $s=t=n$ and bounded $d$, we get the estimate
\begin{equation}
\label{eq-diagonal}
\overline{R}(\mathcal{H},\mathcal{K}^{(3)}_3(n)) \leq 2^{O(n^{2-1/(1+cd^2)})}.
\end{equation}
\end{theorem}

The main idea of the proof of Theorem~\ref{thm-3UnifMaxDegIntChr} is based on an embedding lemma from~\cite{cfs12}, where the authors study Erd\H{o}s--Hajnal-type theorems for $3$-uniform tripartite hypergraphs.
We prove a variant of this lemma, which works for ordered hypergraphs, does not consider induced copies, and uses the assumption that the maximum degree of $\mathcal{H}$ is bounded instead of assuming that the number of vertices of $\mathcal{H}$ is fixed.

Since every ordered $3$-uniform hypergraph $\mathcal{H}$ on $n$ vertices with $\chi_<(\mathcal{H})=3$ is an ordered subhypergraph of~$\mathcal{K}^{(3)}_3(n)$, we immediately obtain the following corollary.

\begin{corollary}
\label{cor-3UnifMaxDegIntChr}
Let $\mathcal{H}$ be an ordered 3-uniform hypergraph on $n$ vertices with maximum degree $d$ and with interval chromatic number $3$.
Then there exists an $\varepsilon = \varepsilon(d)>0$ such that
\[
\overline{R}(\mathcal{H}) \leq 2^{O(n^{2-\varepsilon})}.
\]
\end{corollary}

It might seem wasteful to use Theorem~\ref{thm-3UnifMaxDegIntChr} in order to obtain Corollary~\ref{cor-3UnifMaxDegIntChr}, as the ordered hypergraph $\mathcal{H}$ is much sparser than $\mathcal{K}^{(3)}_3(n)$.
However, as noted in~\cite{clfs17}, this intuition is wrong already for ordered graphs, as there are ordered matchings $\mathcal{M}$ on $n$ vertices with $\chi_<(\mathcal{M})=2$ and ordered graphs $\mathcal{G}$ on $N=2^{n^c}$ vertices for some constant $c>0$ such that $\mathcal{G}$ has edge density at least $1-n^{-c}$ and does not contain $\mathcal{M}$ as an ordered subgraph.
In fact, the best known upper bounds on $\overline{R}(\mathcal{G})$ for $n$-vertex ordered graphs $\mathcal{G}$ with bounded maximum degree and $\chi_<(\mathcal{G})=\chi$ are derived from the bound~\eqref{eq-degreeIntChrNumber} on $\overline{R}(\mathcal{G},\mathcal{K}^{(2)}_\chi(n))$.

The upper bound~\eqref{eq-diagonal} is quite close to the truth, as even when $\mathcal{H}$ is fixed we get a superexponential lower bound on $\overline{R}(\mathcal{H},\mathcal{K}^{(3)}_3(n))$. We note that since the first version of this preprint we learned that independently Fox and He (Theorem 1.3 in~\cite{foxhe19}) proved the same lower bound for the unordered Ramsey number and that implies the result of Theorem \ref{thm-3UnifMaxDegLower}. However we leave this result here as our proof is much simpler. 

\begin{theorem}
\label{thm-3UnifMaxDegLower}
For every $t \geq 3$ and every positive integer $n$, we have
\[\overline{R}(\mathcal{K}^{(3)}_{t+1},\mathcal{K}^{(3)}_3(n)) \geq 2^{\Omega(n \log{n})}.\]
\end{theorem}

We do not have any nontrivial lower bound in the diagonal case $\overline{R}(\mathcal{H})$ for $\mathcal{H}$ with bounded maximum degree and $\chi_<(\mathcal{H})=3$.
We note that even for ordered graphs $\mathcal{G}$ with bounded maximum degree $d$ and $\chi_<(\mathcal{G})=2$ the best known lower bound on $\overline{R}(\mathcal{G})$ (and also on $\overline{R}(\mathcal{H})$) is only of order $\Omega((n/\log{n})^2)$~\cite{bjv16}, while the upper bound on $\overline{R}(\mathcal{G})$ is of order $n^{O(d)}$~\cite{bckk13,clfs17}.

Concerning $k$-uniform hypergraphs with $k>3$, the following result is based on a modification of the proof from~\cite[Proposition~6.3]{cfs10} and gives an estimate on ordered Ramsey numbers of ordered $k$-uniform hypergraphs with bounded interval chromatic number.
In particular, this estimate shows that we do not have a tower-type growth rate for $\overline{R}(\mathcal{H})$ once the uniformity and the interval chromatic number of $\mathcal{H}$ are bounded.

\begin{proposition}
\label{prop-generalBound}
Let $\chi,k$ be integers with $\chi \geq k \geq 2$ and let $\mathcal{H}$ be an ordered $k$-uniform hypergraph on $n$ vertices with interval chromatic number $\chi$.
Then there is a constant $c$ such that
\[\overline{R}(\mathcal{H}) \leq 2^{R^{\chi(\chi-1)}  (c\chi  n)^{\chi-1}},\]
where $R=R(K^{(k)}_\chi)$.
In particular, if the uniformity $k$ and the interval chromatic number $\chi$ of~$\mathcal{H}$ are fixed, we have 
\[
\overline{R}(\mathcal{H}) \leq 2^{O(n^{\chi-1})}.
\]
\end{proposition}

Our understanding of the ordered Ramsey numbers of ordered hypergraphs is still very limited.
Many interesting open problem arose during our study and we would like to draw attention to some of them in Section~\ref{sec-openProblems}.

\section{Proof of Theorem~\ref{thm-3UnifMaxDegIntChr}}

Here we prove Theorem~\ref{thm-3UnifMaxDegIntChr} by giving an upper bound on the ordered Ramsey numbers of ordered 3-uniform hypergraphs with bounded maximum degree versus $\mathcal{K}^{(3)}_3(n)$.
We prove the following slightly stronger result.

\begin{theorem}
\label{thm-3UnifMaxDegIntChrPrecise}
Let $\mathcal{H}$ be an ordered 3-uniform hypergraph on $t$ vertices with maximum degree $d$, let $s$ be a positive integer and $\rho \in (0,1/8)$ be a real number.
Then there is a constant $C'$ such that
\[\overline{R}(\mathcal{H},\mathcal{K}^{(3)}_3(s)) \leq t \cdot 2^{C'(s^{3/2} \cdot \rho^{-30d^2} \cdot d^6 +s^2\log{(\frac{1}{1-4\rho})})}.\]
\end{theorem}

First, we show that Theorem~\ref{thm-3UnifMaxDegIntChrPrecise} implies Theorem~\ref{thm-3UnifMaxDegIntChr}.

\begin{proof}[Proof of Theorem~\ref{thm-3UnifMaxDegIntChr}]
We can assume that $s > 8^{2+60d^2}$ by choosing the constant $C=C(d)$ from the statement of Theorem~\ref{thm-3UnifMaxDegIntChr} sufficiently large.
We then choose $\rho = s^{-1/(2+60d^2)}$, which gives $\rho < 1/8$ by our assumption on $s$.
Since $1-z \geq e^{-2z}$ for every $z$ with $0\leq z \leq 1/2$, we obtain $(1-4\rho)^{-s^2} \leq e^{8\rho s^2}$ and thus $s^2\log{(\frac{1}{1-4\rho})}) \leq 8s^2\rho$.
By our choice of~$\rho$, we then get
\[s^{3/2}\rho^{-30d^2} = s^2\rho = s^{2-1/(2+60d^2)}.\]
Therefore we can rewrite the upper bound from Theorem~\ref{thm-3UnifMaxDegIntChrPrecise} as \[\overline{R}(\mathcal{H},\mathcal{K}^{(3)}_3(s)) \leq t \cdot 2^{C'(d^6 + 8)s^2\rho} \leq 2^{Cs^{2-1/(2+60d^2)}}\]
for sufficiently large $C=C(d)$,
which concludes the proof.
\end{proof}

In the rest of the section, we prove Theorem~\ref{thm-3UnifMaxDegIntChrPrecise}, but we first state some definitions.
For a graph $G$ and two disjoint subsets $X$ and $Y$ of vertices of $G$, we use $d(X,Y)$ to denote the \emph{edge density} between $X$ and $Y$, that is, $d(X,Y) = \frac{e(X,Y)}{|X||Y|}$, where $e(X,Y)$ denotes the number of edges with one vertex in $X$ and with the other one in $Y$.

For positive real numbers $\varepsilon_1,\varepsilon_2$, and $\rho$, we say that an ordered bipartite graph $\mathcal{G}$ with consecutive color classes $U$ and $V$ is \emph{bi-$(\varepsilon_1,\varepsilon_2,\rho)$-dense between $U$ and $V$}, if for all sets $X\subseteq U$ and $Y \subseteq V$ with $|X| \geq \varepsilon_1 |U|$ and $|Y| \geq \varepsilon_2 |V|$ we have $d(X,Y) \geq \rho$.

For positive real numbers $\varepsilon$ and $\rho$ and a positive integer $m$, an ordered 3-uniform hypergraph $\mathcal{H}$ is \emph{tri-$(\varepsilon,\rho,m)$-dense}, if for all consecutive subsets $V_1,V_2,V_3$ of vertices of $\mathcal{H}$, each of size at most $m$, and for all bipartite graphs $G_{1,2},G_{1,3},G_{2,3}$, each $G_{i,j}$ between $V_i$ and $V_j$, for which there are at least $\varepsilon m^3$ triangles with one edge in each $G_{i,j}$, at least $\rho$-proportion of these triangles forms hyperedges in $\mathcal{H}$.

The following embedding lemma is based on a similar result from~\cite{cfs12}.

\begin{lemma}\label{lem-embedding}
Let $\mathcal{H}$ be an ordered 3-uniform hypergraph on $t$ vertices with maximum degree $d$.
Let $ \varepsilon >0$ and $\rho \in (0,1)$ be two real numbers with $ \varepsilon \leq 2^{-6} \cdot \rho^{15d^2} \cdot d^{-3}$.
If $\mathcal{G}$ is a tri-$(\varepsilon,\rho,n/t)$-dense ordered 3-uniform hypergraph on $n \geq t/\varepsilon$ vertices, then $\mathcal{H}$ is an ordered subhypergraph of~$\mathcal{G}$. 
\end{lemma}
\begin{proof}
Let $v_1 \prec \cdots \prec v_t$ be the vertices of $\mathcal{H}=(H,\prec)$.
For a vertex $v_j$ of $\mathcal{H}$ with $j>i$, we use ${\rm deg}^{{\preceq}i}_\mathcal{H}(j)$ to denote the number $\sum_{e \in E(\mathcal{H}): v_j \in e} |e \cap \{v_1,\dots,v_i\}|$.
Similarly, for vertices $v_j$ and $v_k$ of $\mathcal{H}$ and $i<j,k$, we use ${\rm deg}^{{\preceq}i}_\mathcal{H}(j,k)$ to denote the number of hyperedges of $\mathcal{H}$ containing $v_j$, $v_k$, and a vertex from $\{v_1,\dots,v_i\}$.
Observe that ${\rm deg}^{{\preceq}i}_\mathcal{H}(j) \leq 2d$  and ${\rm deg}^{{\preceq}i}_\mathcal{H}(j,k) \leq d$ for all $j$ and $k$ with $i<j,k$, as the maximum degree of $\mathcal{H}$ is $d$ and each hyperedge is multiplied by at most $2$.

We partition the vertex set of $\mathcal{G}$ into consecutive intervals $U_1,\dots,U_t$, each of size $n/t$.
We embed a copy $f(\mathcal{H})$ of $\mathcal{H}$ in $\mathcal{G}$ one vertex at time, embedding each vertex $f(v_i)$ of $f(\mathcal{H})$ to~$U_i$.
We proceed by induction on $i=0,1,\dots,t-1$.
Assuming that the vertices $v_1,\dots,v_i$ have been embedded as $f(v_1),\dots,f(v_i)$, we show that there are sets $U^i_{i+1},\dots,U^i_t$ and \allowbreak graphs $G^i_{j,k}$ with $i<j<k\leq t$ such that the following three conditions are satisfied:
\begin{enumerate}[label=(\roman*)]
    \item\label{item-embed1} $|U^i_j| \geq C^i_j \cdot n/t$, where $C^i_j = \rho^{d \cdot {\rm deg}^{{\preceq}i}_\mathcal{H}(j)}$,
    \item\label{item-embed2} $G^i_{j,k}$ is a bipartite graph between $U^i_j$ and $U^i_k$, which is bi-$(\varepsilon^i_j,\varepsilon^i_k,\rho^i_{j,k})$-dense between $U^i_j$ and $U^i_k$ with $\varepsilon^i_j = \rho^{4d^2 - d \cdot {\rm deg}^{{\preceq}i}_\mathcal{H}(j)}/(4d)$, $\varepsilon^i_k = \rho^{4d^2 - d \cdot {\rm deg}^{{\preceq}i}_\mathcal{H}(k)}/(4d)$, and $\rho^i_{j,k} = \rho^{{\rm deg}^{{\preceq}i}_\mathcal{H}(j,k)}$, and
    
    \item\label{item-embed3} for every $h \leq i$, every edge of $G^i_{j,k}$ forms a hyperedge of $\mathcal{G}$ with $f(v_h)$ if $\{v_h,v_j,v_k\}$ is a hyperedge of $\mathcal{H}$. Also, for all $h_1$ and $h_2$ with $h_1 < h_2 \leq i$, every vertex $u \in U^i_j$ with $i<j$ forms a hyperedge $\{ f(v_{h_1}),f(v_{h_2}),u\}$ of $\mathcal{G}$ if $\{v_{h_1},v_{h_2},v_j\}$ is a hyperedge of $\mathcal{H}$.
\end{enumerate}

For the induction base, assume $i=0$ and set $U^0_j = U_j$ for every $j \in \{1,\dots,t\}$.
For all $j$ and $k$ with $1 \leq j<k\leq t$, we let $G^i_{j,k}$ be the complete bipartite graph between $U^0_j$ and $U^0_k$.
Then the three conditions are trivially satisfied.

For the induction step, assume that the vertices $v_1,\dots,v_i$ have been embedded for some $i \geq 0$ while maintaining conditions~\ref{item-embed1}, \ref{item-embed2}, and \ref{item-embed3}.
We show how to embed the vertex $v_{i+1}$.

We let $W_{i+1}$ be the set of vertices $w$ from $U^i_{i+1}$ such that the neighborhood $U^i_j(w)$ of $w$ in $G^i_{i+1,j}$ has size at least $\rho^i_{i+1,j} |U^i_j|$ for every $j \in \{i+2,\dots,t\}$ such that $v_{i+1}$ and $v_j$ are contained in a hyperedge of $\mathcal{H}$.
Since the graph $G^i_{i+1,j}$ is bi-$(\varepsilon^i_{i+1},\varepsilon^i_j,\rho^i_{i+1,j})$-dense, there are at most $\varepsilon^i_{i+1}|U^i_{i+1}| \leq  \varepsilon^i_{i+1} n/t$ vertices in $U_{i+1}$ that have less than $\rho^i_{i+1,j} |U^i_j|$ neighbors in $|U^i_j|$.
Using condition~\ref{item-embed1} and the fact that the maximum degree of $\mathcal{H}$ is $d$, we see that the size of~$W_{i+1}$ satisfies
\begin{align*}
|W_{i+1}| &\geq |U^i_{i+1}| - 2d\varepsilon^i_{i+1}|U^i_{i+1}| \geq |U^i_{i+1}| - 2d\varepsilon^i_{i+1} n/t \geq C^i_{i+1} n/t - 2d\varepsilon^i_{i+1} n/t \geq C^i_{i+1} n/(2t),
\end{align*}
where the last inequality follows from $\varepsilon^i_{i+1} \leq C^i_{i+1}/(4d)$, as ${\rm deg}^{{\preceq}i}_\mathcal{H}(i+1) \leq 2d$.

For a vertex $w \in W_{i+1}$ and indices $j$ and $k$ such that $i+1<j<k$ and $\{v_{i+1},v_j,v_k\} \in E(\mathcal{H})$, we define the graph $H_{j,k}(w)$ as a subgraph of $G^i_{j,k}$ between $U^i_j(w)$ and $U^i_k(w)$ consisting of edges $\{x,y\}$ such that $\{w,x,y\} \in E(\mathcal{G})$.
We also let $W^{i+1}_{j,k}$ be the set of vertices $w \in W_{i+1}$ such that the graph $H_{j,k}(w)$ is not bi-$(\varepsilon^{i+1}_j,\varepsilon^{i+1}_k,\rho^{i+1}_{j,k})$-dense between $U^i_j(w)$ and $U^i_k(w)$.

By the definition of bi-$(\varepsilon^{i+1}_j,\varepsilon^{i+1}_k,\rho^{i+1}_{j,k})$-density, for every $w \in W^{i+1}_{j,k}$, there are sets $Y_j(w) \subseteq U^i_j(w)$ and $Y_k(w) \subseteq U^i_k(w)$ such that $|Y_j(w)| \geq \varepsilon^{i+1}_j|U^i_j(w)|$, $|Y_k(w)| \geq \varepsilon^{i+1}_k|U^i_k(w)|$, and $d(Y_j(w),Y_k(w)) < \rho^{i+1}_{j,k}$ in $H_{j,k}(w)$.
Since $w \in W_{i+1}$, we have
\[|Y_j(w)| \geq \varepsilon^{i+1}_j|U^i_j(w)| \geq \varepsilon^{i+1}_j\cdot \rho^i_{i+1,j} |U^i_j| \geq \varepsilon^i_j |U^i_j|,\]
where the last inequality follows from $\varepsilon^{i+1}_j\cdot \rho^i_{i+1,j} \geq \varepsilon^i_j$, as, since $\{v_{i+1},v_j,v_k\} \in E(\mathcal{H})$, we have ${\rm deg}_{\mathcal{H}}^{{\preceq}i+1}(j) > {\rm deg}_{\mathcal{H}}^{{\preceq}i}(j)$ and ${\rm deg}^{{\preceq}i}_\mathcal{H}(i+1,j) \leq d$.
Analogously, we obtain $|Y_k(w)| \geq \varepsilon^i_k |U^i_k|$.

We let $J_{i+1,j}$ be the graph connecting each $w \in W^{i+1}_{j,k}$ to vertices from $Y_j(w)$ and we analogously define the graph $J_{i+1,k}$.
The graph $G^i_{j,k}$ is bi-$(\varepsilon^i_j,\varepsilon^i_k,\rho^i_{j,k})$-dense between $U^i_j$ and $U^i_k$ by condition~\ref{item-embed2}.
Thus, since $|Y_j(w)| \geq \varepsilon^i_j |U^i_j|$ and $|Y_k(w)| \geq \varepsilon^i_k |U^i_k|$, the number of triangles in the tripartite graph between $W^{i+1}_{j,k}$, $U^i_j$, and $U^i_k$ formed by $J_{i+1,j}$, $J_{i+1,k}$, and $G^i_{j,k}$ is at least $\rho^i_{j,k}\sum_{w \in W^{i+1}_{j,k}}|Y_j(w)||Y_k(w)|$.

Suppose for contradiction that $|W^{i+1}_{j,k}| \geq |W_{i+1}|/(2d)$ for some $j$ and $k$ with $i+1<j<k$ and $\{v_{i+1},v_j,v_k\}\in E(\mathcal{H})$.
Then, since $|Y_j(w)| \geq \varepsilon^i_j |U^i_j|$ and $|Y_k(w)| \geq \varepsilon^i_k |U^i_k|$ for every $w \in W^{i+1}_{j,k}$, 
\[\rho^i_{j,k}\sum_{w \in W^{i+1}_{j,k}}|Y_j(w)||Y_k(w)| \geq \rho^i_{j,k}\varepsilon^i_j\varepsilon^i_k |W^{i+1}_{j,k}||U^i_j||U^i_k| \geq \rho^i_{j,k}\varepsilon^i_j\varepsilon^i_k C^i_jC^i_k |W^{i+1}_{j,k}| (n/t)^2,\]
where the last inequality follows from condition~\ref{item-embed1}.
Using the assumption $|W^{i+1}_{j,k}| \geq |W_{i+1}|/(2d)$ and the fact $|W_{i+1}| \geq C^i_{i+1}n/(2t)$, we can estimate the above expression from below by 
\[\rho^i_{j,k}\varepsilon^i_j\varepsilon^i_k C^i_jC^i_k C^i_{i+1} n^3/(4dt^3)\]
Our choice of parameters then gives $\rho^i_{j,k}\varepsilon^i_j\varepsilon^i_k C^i_jC^i_k C^i_{i+1}/(4dt^3) \geq \varepsilon/t^3$.
Altogether, there are at least 
\[\rho^i_{j,k}\sum_{w \in W^{i+1}_{j,k}}|Y_j(w)||Y_k(w)| \geq \varepsilon (n/t)^3\] 
triangles in the tripartite graph between $W^{i+1}_{j,k}$, $U^i_j$, and $U^i_k$ formed by $J_{i+1,j}$, $J_{i+1,k}$, and $G^i_{j,k}$.
Thus, since the ordered hypergraph $\mathcal{G}$ is tri-$(\varepsilon,\rho,n/t)$-dense, at least $\rho$-proportion of these triangles forms hyperedges in $\mathcal{G}$. 
Therefore there are at least $\rho \cdot \rho^i_{j,k}\sum_{w \in W^{i+1}_{j,k}}|Y_j(w)||Y_k(w)|$ hyperedges of $\mathcal{G}$ between $W^{i+1}_{j,k}$, $U^i_j$, and $U^i_k$.

On the other hand, the number of hyperedges of $\mathcal{G}$ containing a vertex $w \in W^{i+1}_{j,k}$ and having an edge in each of the graphs $J_{i+1,j}$, $J_{i+1,k}$, and $G^i_{j,k}$ is the number of edges of~$H_{j,k}(w)$ between $Y_j(w)$ and $Y_k(w)$.
This number of edges is less than $\rho^{i+1}_{j,k}|Y_j(w)||Y_k(w)|$, as we know that $d(Y_j(w),Y_k(w)) < \rho^{i+1}_{j,k}$ in $H_{j,k}(w)$.
Thus the number of the hyperedges of $\mathcal{G}$ is less than 
\[\rho^{i+1}_{j,k}\sum_{w \in W^{i+1}_{j,k}} |Y_j(w)||Y_k(w)| \leq \rho \cdot \rho^i_{j,k}\sum_{w \in W^{i+1}_{j,k}} |Y_j(w)||Y_k(w)|,\]
where the least inequality follows from $\rho^{i+1}_{j,k} \leq \rho \cdot \rho^i_{j,k}$, as $\{v_{i+1},v_j,v_k\} \in E(\mathcal{H})$ and thus we have ${\rm deg}_{\mathcal{H}}^{{\preceq}i+1}(j,k) > {\rm deg}_{\mathcal{H}}^{{\preceq}i}(j,k)$.
This contradicts the fact that there are at least $\rho~\cdot~ \rho^i_{j,k}\sum_{w \in W^{i+1}_{j,k}} |Y_j(w)||Y_k(w)|$ such hyperedges  of $\mathcal{G}$.

Thus $|W^{i+1}_{j,k}| < |W_{i+1}|/(2d)$.
In particular, the number of vertices $w \in W_{i+1}$ that do not lie in any set $W^{i+1}_{j,k}$ such that $i+1 < j <k \leq t$ and $\{v_{i+1},v_j,v_k\} \in E(\mathcal{H})$ is at least
\[|W_{i+1}| - \sum_{j<k: \{v_{i+1},v_j,v_k\} \in E(\mathcal{H})} |W^{i+1}_{j,k}| > |W_{i+1}| - \frac{d |W_{i+1}|}{2d} = \frac{|W_{i+1}|}{2},\]
as we are summing over at most $d$ pairs $(j,k)$, because the maximum degree of $\mathcal{H}$ is $d$.
Since $|W_{i+1}| \geq C^i_{i+1}n/(2t)$, we have at least $C^i_{i+1}n/(4t)$ such vertices and we let $f(v_{i+1})$ be any of them.
Since $n \geq t/\varepsilon > 4t/C^i_{i+1}$, at least one such vertex indeed exists.
For every $j \in \{i+2,\dots,t\}$ such that $v_{i+1}$ and $v_j$ are contained in a hyperedge of $\mathcal{H}$, we let $U^{i+1}_j=U^i_j(f(v_{i+1}))$.
We keep $U^{i+1}_j=U^i_j$ for all other values $j$.
Let $j$ and $k$ be indices such that $i+1 < j < k \leq t$.
If $\{v_{i+1},v_j,v_k\} \in E(\mathcal{H})$, we set $G^{i+1}_{j,k}=H_{j,k}(f(v_{i+1}))$.
For all other values of $j$ and $k$, we let $G^{i+1}_{j,k}$ be the subgraph of $G^i_{j,k}$ induced by $U^{i+1}_j$ and $U^{i+1}_k$.
In particular, if none of the vertices $v_j$ and $v_k$ lies in a hyperedge of $\mathcal{H}$ with $v_{i+1}$, we have $G^{i+1}_{j,k} = G^i_{j,k}$.

To finish the induction step, it remains to verify conditions~\ref{item-embed1}, \ref{item-embed2}, and~\ref{item-embed3}.
To verify condition~\ref{item-embed1}, first observe that if, for $j \in \{i+2,\dots,k\}$, the vertex $v_j$ is not in a hyperedge of $\mathcal{G}$ with $v_{i+1}$, then, by the choice of $C^i_j$ and $C^{i+1}_j$, $|U^{i+1}_j| = |U^i_j| \geq C^i_j n/t =C^{i+1}_j n/t$, as ${\rm deg}^{{\preceq}i}_\mathcal{H}(j) = {\rm deg}^{{\preceq}i+1}_\mathcal{H}(j)$.
Otherwise $|U^{i+1}_j| = |U^i_j(f(v_{i+1}))|$.
Since $f(v_{i+1}) \in W_{i+1}$, we have $|U^i_j(f(v_{i+1}))| \geq \rho^i_{i+1,j}|U^i_j|$.
Since ${\rm deg}^{{\preceq}i}_\mathcal{H}(j) < {\rm deg}^{{\preceq}i+1}_\mathcal{H}(j)$ and ${\rm deg}^{{\preceq}i}_\mathcal{H}(i+1,j) \leq d$, we have $\rho^i_{i+1,j}C^i_j \geq C^{i+1}_j$.
So we obtain $|U^{i+1}_j| \geq C^{i+1}_j n/t$.
Thus condition~\ref{item-embed1} is satisfied.

For condition~\ref{item-embed2}, let $j$ and $k$ be indices such that $i+1 < j<k\leq t$.
Consider first the case when $\{v_{i+1},v_j,v_k\} \in E(\mathcal{H})$.
Then $G^{i+1}_{j,k} = H_{j,k}(f(v_{i+1}))$.
Since $f(v_{i+1})$ does not lie in any set $W^{i+1}_{j,k}$ such that $i+1<j<k$ and $\{v_{i+1},v_j,v_k\} \in E(\mathcal{H})$, the graph $H_{j,k}(f(v_{i+1}))$ is bi-$(\varepsilon^{i+1}_j,\varepsilon^{i+1}_k,\rho^{i+1}_{j,k})$-dense between $U^i_j(f(v_{i+1}))$ and $U^i_k(f(v_{i+1}))$.
Thus $G^{i+1}_{j,k} = H_{j,k}(f(v_{i+1}))$ is bi-$(\varepsilon^{i+1}_j,\varepsilon^{i+1}_k,\rho^{i+1}_{j,k})$-dense between $U^{i+1}_j=U^i_j(f(v_{i+1}))$ and $U^{i+1}_k=U^i_k(f(v_{i+1}))$, which verifies condition~\ref{item-embed2} in this case.
Now, assume $\{v_{i+1},v_j,v_k\} \notin E(\mathcal{H})$.
If none of the two vertices $v_j$ and $v_k$ is in a hyperedge of $\mathcal{H}$ with $v_{i+1}$, then ${\rm deg}^{{\preceq}i+1}_\mathcal{H}(j) = {\rm deg}^{{\preceq}i}_\mathcal{H}(j)$, ${\rm deg}^{{\preceq}i+1}_\mathcal{H}(k) = {\rm deg}^{{\preceq}i}_\mathcal{H}(k)$, and ${\rm deg}^{{\preceq}i+1}_\mathcal{H}(j,k) = {\rm deg}^{{\preceq}i}_\mathcal{H}(j,k)$.
In particular, $\varepsilon^i_j = \varepsilon^{i+1}_j$, $\varepsilon^i_j = \varepsilon^{i+1}_j$, and $\rho^{i+1}_{j,k} = \rho^i_{j,k}$.
Since $U^{i+1}_j=U^i_j$, $U^{i+1}_k=U^i_k$, and $G^{i+1}_{j,k} = G^i_{j,k}$ is bi-$(\varepsilon^i_j,\varepsilon^i_k,\rho^i_{j,k})$-dense between $U^i_j$ and $U^i_k$, we see that $G^{i+1}_{j,k}$ is bi-$(\varepsilon^{i+1}_j,\varepsilon^{i+1}_k,\rho^{i+1}_{j,k})$-dense between $U^{i+1}_j$ and $U^{i+1}_k$, which again verifies condition~\ref{item-embed2}.
It remains to consider the case when exactly one of the vertices $v_j$ and $v_k$ is in a hyperedge of $\mathcal{H}$ with $v_{i+1}$.
By symmetry, we can assume without loss of generality that $v_j$ and $v_{i+1}$ are in a hyperedge of~$\mathcal{H}$.
Then  ${\rm deg}^{{\preceq}i+1}_\mathcal{H}(j) > {\rm deg}^{{\preceq}i}_\mathcal{H}(j)$, ${\rm deg}^{{\preceq}i+1}_\mathcal{H}(k) = {\rm deg}^{{\preceq}i}_\mathcal{H}(k)$, and ${\rm deg}^{{\preceq}i+1}_\mathcal{H}(j,k) = {\rm deg}^{{\preceq}i}_\mathcal{H}(j,k)$.
In particular, $\varepsilon^i_j > \varepsilon^{i+1}_j$, $\varepsilon^i_k = \varepsilon^{i+1}_k$, and $\rho^{i+1}_{j,k} = \rho^i_{j,k}$.
By definition, we have $U^{i+1}_j=U^i_j(f(v_{i+1}))$ and $U^{i+1}_k=U^i_k$.
Let $X$ be a subset of $U^{i+1}_j$ of size at least $\varepsilon^{i+1}_j|U^{i+1}_j|$ and let $Y$ be a subset of $U^{i+1}_k$ of size at least $\varepsilon^{i+1}_k|U^{i+1}_k| = \varepsilon^i_k|U^i_k|$.
We want to show that $d(X,Y) \geq \rho^{i+1}_{j,k}$ in $G^{i+1}_{j,k}$.
The subset $X$ has size at least $\varepsilon^i_j |U^i_j|$, as $|U^{i+1}_j| = |U^i_j(f(v_{i+1}))| \geq \rho^i_{i+1,j}|U^i_j|$ and our choice of $\varepsilon^{i+1}_j$ together with the fact ${\rm deg}^{{\preceq}i}_\mathcal{H}(i+1,j) \leq d$ gives $\varepsilon^{i+1}_j|U^{i+1}_j| \geq \varepsilon^i_j |U^i_j|$.
Thus, since $G^i_{j,k}$ is bi-$(\varepsilon^i_j,\varepsilon^i_k,\rho^i_{j,k})$-dense between $U^i_j \supseteq U^{i+1}_j$ and $U^i_k \supseteq U^{i+1}_k$, the density between $X$ and $Y$ is at least $\rho^i_{j,k}=\rho^{i+1}_{j,k}$ in $G^i_{j,k}$.
Since $G^{i+1}_{j,k}$ is an induced subgraph of $G^i_{j,k}$, the density between $X$ and $Y$ is also at least $\rho^i_{j,k}=\rho^{i+1}_{j,k}$ in $G^{i+1}_{j,k}$. 
Thus $G^{i+1}_{j,k}$ is bi-$(\varepsilon^{i+1}_j,\varepsilon^{i+1}_k,\rho^{i+1}_{j,k})$-dense between $U^{i+1}_j$ and $U^{i+1}_k$, which verifies condition~\ref{item-embed2}.

We show that Condition~\ref{item-embed3} is satisfied as well.
Let $j$ and $k$ be indices such that $i+2<j<k\leq t$.
If $\{v_h,v_j,v_k\}$ is a hyperedge of $\mathcal{H}$ for some $h \leq i+1$, then $\{f(v_h),x,y\}$ is a hyperedge of $\mathcal{G}$ for every edge $\{x,y\}$ of $G^{i+1}_{j,k}$.
This is true for $h \leq i$ by the induction assumption, as $G^{i+1}_{j,k} \subseteq G^i_{j,k}$.
For $h=i+1$ we have $G^{i+1}_{j,k} = H_{j,k}(f(v_{i+1}))$ and the claim holds by the definition of $H_{j,k}(f(v_{i+1}))$.
Similarly, for all $h_1$ and $h_2$ with $h_1 < h_2 \leq i+1$, every vertex $u \in U^{i+1}_j$ with $i+1<j$ forms a hyperedge $\{f(v_{h_1}),f(v_{h_2}),u\}$ of $\mathcal{G}$ if $\{v_{h_1},v_{h_2},v_j\}$ is a hyperedge of $\mathcal{H}$.
This is because if $h_2\leq i$, then the claim follows from the induction assumption and the fact $U^{i+1}_j \subseteq U^i_j$.
For $h_2=i+1$, the triple $\{f(v_{h_1}),f(v_{h_2}),x\}$ is a hyperedge of $\mathcal{G}$ for every $x \in U^{i+1}_j = U^i_j(f(v_{h_2}))$ by the inductive assumption.

Finally, after we find all vertices $f(v_1),\dots,f(v_t)$, condition~\ref{item-embed3} ensures that they determine a copy of $\mathcal{H}$ as an ordered subhypergraph of $\mathcal{G}$.
\end{proof}

We use the following result proved by Conlon, Fox, and Sudakov~\cite{cfs12}.
It says that if a graph $G$ contains many triangles, a 3-uniform hypergraph whose hyperedges form a dense subset of the set of triangles in $G$ contains a large copy of $K^{(3)}_3(n)$.

\begin{lemma}[\cite{cfs12}]
\label{lem-tripartite}
Let $V_1,V_2,V_3$ be pairwise disjoint sets of vertices, each of size at most $m$, and let $G_{i,j}$ be a bipartite graph between $V_i$ and $V_j$ for all $i$ and $j$ with $1 \leq i < j \leq 3$.
Assume there are at least $\delta m^3$ triangles in the tripartite graph formed by $G_{1,2}$, $G_{1,3}$, and $G_{2,3}$.
Let $G$ be a $3$-uniform hypergraph containing at least $(1-\eta)$-proportion of the triangles in the tripartite graph, where $0<\eta < 1/8$.
Then $G$ contains a copy of $K^{(3)}_3(s)$ provided that
\[e^{2^{10}\delta^{-2}s^{3/2}}(1-4\eta)^{-4s^2}\left(\frac{\delta}{16}\right)^{4s} \leq m.\]
\end{lemma}

We note that, although it is not explicitly stated in the above lemma, the copy of $K^{(3)}_3(s)$ has $V_i$ as the $i$th color class.
Thus if the set $V_1 \cup V_2 \cup V_3$ is ordered and $V_1,V_2,V_3$ are consecutive in this ordering, we actually get a copy of the ordered hypergraph $\mathcal{K}^{(3)}_3(s)$.

We now proceed with the proof of Theorem~\ref{thm-3UnifMaxDegIntChrPrecise}, which implies Theorem~\ref{thm-3UnifMaxDegIntChr}.

\begin{proof}[Proof of Theorem~\ref{thm-3UnifMaxDegIntChrPrecise}]
Let $\mathcal{H}$ be an ordered 3-uniform hypergraph on $t$ vertices with maximum degree $d$, let $s$ be a positive integer and $\rho \in (0,1/8)$ be a real number.
We choose $\varepsilon = 2^{-6} \cdot \rho^{15d^2} \cdot d^{-3}$ and we let $N$ be an integer such that 
\[N \geq t \cdot 2^{2^{28}\cdot s^{3/2} \cdot \rho^{-30d^2} \cdot d^6 +12s^2\log{(\frac{1}{1-4\rho})}}.\]
Note that $N \geq t/\varepsilon$.
Consider a red-blue coloring $\chi$ of the hyperedges of~$\mathcal{K}^{(3)}_N$.
We use $\mathcal{G}$ to denote the ordered $3$-uniform hypergraph on $N$ vertices formed by the hyperedges of~$\mathcal{K}^{(3)}_N$ that are blue in~$\chi$.
Similarly, we let $\overline{\mathcal{G}}$ be the hypergraph determined by red hyperedges of~$\mathcal{K}^{(3)}_N$ in~$\chi$.

If $\mathcal{G}$ is tri-$(\varepsilon,\rho,N/t)$-dense, then, since $N \geq t/\varepsilon$, Lemma~\ref{lem-embedding} implies that there is a blue copy of~$\mathcal{H}$ in $\chi$ and we are done.
Thus we assume that $\mathcal{G}$ is not tri-$(\varepsilon,\rho,N/t)$-dense.
That is, there are three consecutive subsets $V_1,V_2,V_3$ of vertices of $\mathcal{G}$, each of size at most $N/t$, and three bipartite graphs $G_{1,2},G_{1,3},G_{2,3}$, each $G_{i,j}$ between $V_i$ and $V_j$, for which there are at least $\varepsilon (N/t)^3$ triangles with one edge in each $G_{i,j}$ and less than $\rho$-proportion of these triangles forms hyperedges in $\mathcal{G}$.
By the choice of $\mathcal{G}$, at least $(1-\rho)$-proportion of these triangles forms hyperedges in $\overline{\mathcal{G}}$.

We show that \begin{equation}
\label{eq-tripartite}
e^{2^{10}\varepsilon^{-2}s^{3/2}}\cdot (1-4\rho)^{-4s^2}\cdot \left(\frac{16}{\varepsilon}\right)^{4s} \leq N/t.
\end{equation}
Then, by Lemma~\ref{lem-tripartite} applied with $\delta = \varepsilon$, $\eta = \rho$, and $m=N/t$, the ordered hypergraph $\overline{\mathcal{G}}$ contains a copy of $K^{(3)}_3(s)$ as an ordered subhypergraph.
By the definition of $\overline{\mathcal{G}}$, all hyperedges of such copy are red in $\chi$, which then finishes the proof.

To estimate~\eqref{eq-tripartite}, 
we estimate each of the three terms in the above expression separately by $N^{1/3}$.
The exponent in the first term in~\eqref{eq-tripartite} is
\[\log{(e)}2^{10}\varepsilon^{-2}s^{3/2} \leq 2^{23}\cdot \rho^{-30d^2}\cdot d^6 \cdot s^{3/2} \leq \frac{\log{(N/t)}}{3}.\]
The second term can be estimated by
\[(1-4\rho)^{-4 s^2} = 2^{-4s^2\log{(1-4\rho)}} \leq (N/t)^{1/3}.\]
Finally, the the third term satisfies
\[\left(\frac{16}{\varepsilon}\right)^{4s} = \left(2^{10}\cdot \rho^{-15d^2} \cdot d^3 \right)^{4s} < 2^{26\cdot s \cdot \rho^{-15d^2} \cdot d^3} \leq (N/t)^{1/3},\]
as $(2^{10}x)^4 < 2^{26x}$ for $x \geq 2$.
Altogether, both inequalities in~\eqref{eq-tripartite} are satisfied.
\end{proof}

\section{Proof of Theorem~\ref{thm-3UnifMaxDegLower}}

Here we prove a superexponential lower bound on $\overline{R}(\mathcal{K}^{(3)}_t,\mathcal{K}^{(3)}_3(n))$ for any $t \geq 4$.
First, we prove the following result, which gives superexponential lower bounds on ordered Ramsey numbers of 3-uniform hypergraphs $\mathcal{K}^{(3)}_{t+1}$ and $\mathcal{K}^{(3)}_3(n)$ provided that we have a superlinear lower bound on the ordered Ramsey number $\overline{R}(\mathcal{K}_t,\mathcal{K}_{m,m})$ in $m$.
The proof is inspired by the approach from~\cite{cfs10}.

\begin{lemma}
\label{lem-3UnifMaxDegLowerPrecise}
For $t \in \mathbb{N}$, if for $m = \lceil n/4\rceil$ we have $2n < \overline{R}(\mathcal{K}^{(3)}_t,\mathcal{K}_{m,m})^\alpha$ for some $\alpha \in (0,1]$,
then, for each sufficiently large $n$,
\[\overline{R}(\mathcal{K}^{(3)}_{t+1},\mathcal{K}^{(3)}_3(n)) \geq \left(\frac{(\overline{R}(\mathcal{K}_t,\mathcal{K}_{m,m})-1)^{1-\alpha}}{e^4}\right)^{(n+2)/6}.\]
\end{lemma}
\begin{proof}
We set $m = \lceil n/4 \rceil$ and $\ell = \lceil n/2 \rceil$.
Let $R$ denote the number $\overline{R}(\mathcal{K}_t,\mathcal{K}_{m,m})-1$ and let $N = (R^{1-\alpha}/e^4)^{(n+2)/6}$.
Using a probabilistic argument, we find a red-blue coloring $\chi$ of the hyperedges of~$\mathcal{K}^{(3)}_N=(K^{(3)}_N,<)$ that does not contain a blue copy of $\mathcal{K}^{(3)}_{t+1}$ and with positive probability does not contain a red copy of~$\mathcal{K}^{(3)}_3(n)$.

We use the following two auxiliary colorings.
Let $\chi_1$ be a red-blue coloring of the edges of~$\mathcal{K}_R$ that does not contain a blue copy of $\mathcal{K}_t$ nor a red copy of $\mathcal{K}_{m,m}$.
Such a coloring exists by the choice of $R$.
Let $\chi_2$ be a coloring of the edges of $\mathcal{K}_N$ with colors $1,\dots,R$, where the color of each edge is chosen uniformly independently at random from~$[R]$.
For three vertices $u < v < w$ of $\mathcal{K}^{(3)}_N$, we then set $\chi(u,v,w) = \chi_1(\chi_2(u,v),\chi_2(u,w))$ if $\chi_2(u,v)\neq \chi_2(u,w)$ and we let $\chi(u,v,w)$ be red otherwise.

Suppose for contradiction that $\chi$ contains a blue copy of $\mathcal{K}^{(3)}_{t+1}$ on some vertices $v_0, \dots, v_t$ of~$\mathcal{K}^{(3)}_N$.
Since all hyperedges of the copy of $\mathcal{K}^{(3)}_{t+1}$ are blue, our choice of $\chi$ gives distinct colors $\chi_2(v_0,v_1),\dots,\chi_2(v_0,v_t)$ and each edge $\{\chi_2(v_0,v_i),\chi_2(v_0,v_j)\}$ such that $\{v_0,v_i,v_j\}$ is a hyperedge of the blue copy of $\mathcal{K}^{(3)}_{t+1}$ is blue in $\chi_1$.
These blue edges determine a blue copy of~$\mathcal{K}_t$ in~$\chi_1$, which contradicts our choice of~$\chi_1$.

We now show that with positive probability there is no red copy of $\mathcal{K}^{(3)}_3(n)$ in $\chi$ by estimating the expected number of such copies.
Consider $3n$ vertices $v_1 < \cdots < v_{3n}$ of~$\mathcal{K}^{(3)}_N$ and suppose that these vertices induce a red copy of $\mathcal{K}^{(3)}_3(n)$ in $\chi$.
We fix $i \in [n]$.
Then there cannot be $2\ell$ distinct colors with $\ell$ of them among the colors $\chi_2(v_i,v_j)$ for $n<j \leq 2n$ and with the remaining $\ell$ of them among $\chi_2(v_i,v_k)$ for $2n < k \leq 3n$. 
Otherwise such colors form a red copy of some ordering of $K_{\ell,\ell}$ in $\chi_1$ and every such ordering contains a copy of $\mathcal{K}_{m,m}$, as $\ell \geq 2m-1$. 
Thus either all the colors $\chi_2(v_i,v_j)$ with $n < j \leq 2n$ are contained in the union of the set $\{\chi_2(v_i,v_k) \colon 2n<k \leq 3n\}$ together with a set of at most $\ell-1$ additional colors or, similarly, all the colors $\chi_2(v_i,v_k)$ with $2n < k \leq 3n$ are contained in the union of the set $\{\chi_2(v_i,v_j) \colon n<j \leq 2n\}$ together with a set of at most $\ell-1$ additional colors.

In the first case, all colors $\chi_2(v_i,v_j)$ with $n<j \leq 2n$ are contained in the union of the set $\{\chi_2(v_i,v_k) \colon 2n<k \leq 3n\}$ together with a set of at most $\ell-1$ additional colors.
In particular, the colors $\chi_2(v_i,v_j)$ with $n<j \leq 2n$ and $\chi_2(v_i,v_k)$ with $2n<k \leq 3n$ are all contained in a set of $n+\ell-1$ colors from $[R]$, as $|\{\chi_2(v_i,v_k) \colon 2n<k \leq 3n\}| \leq n$.
There are $\binom{R}{n+\ell-1}$ sets for the possible colors.
The probability that each of the colors $\chi_2(v_i,v_j)$ with $n<j \leq 2n$ and $\chi_2(v_i,v_k)$ with $2n<k \leq 3n$ is contained in a fixed set of $n+\ell-1$ colors from $[R]$ is $\left(\frac{n+\ell-1}{R}\right)^{2n}$.
The other case is symmetric.
Considering this for every $i \in [n]$, we see that the expected number of red copies of $\mathcal{K}^{(3)}_3(n)$ in~$\chi$ is at most
\begin{align*}
    &\binom{N}{3n}\left(2\binom{R}{n+\ell-1}^n\left(\frac{n+\ell-1}{R}\right)^{2n^2}\right) \\
    &\leq N^{3n}\left(2\left(\frac{eR}{n+\ell-1}\right)^{n(n+\ell-1)}\left(\frac{n+\ell-1}{R}\right)^{2n^2}\right)\\
    &\leq \left(N^3\left(2 e^{n+\ell-1}\left(\frac{n+\ell-1}{R}\right)^{n-\ell+1}\right) \right)^n \\
    &< \left(N^3e^{2n}\left(R^{\alpha-1}\right)^{n-\ell+1}\right)^n\\
    &\leq \left(N^3\left(e^4R^{\alpha-1}\right)^{n/2+1}\right)^n = 1,
\end{align*}
where we used $n+\ell-1\leq 2n \leq R^\alpha$ and our choice of $\ell$ and $N$.
Since the expected number of red copies of $\mathcal{K}^{(3)}_3(n)$ is less than $1$, there is a red-blue coloring $\chi$ that does not contain a blue copy of~$\mathcal{K}^{(3)}_{t+1}$ nor a red copy of~$\mathcal{K}^{(3)}_3(n)$.
\end{proof}

We also use the following bounds on Ramsey numbers of complete graphs and complete bipartite graphs proved by Li and Zang~\cite{liZan03}.

\begin{theorem}[\cite{liZan03}]
\label{thm-ramCliqueBipartite}
For every fixed integer $t \geq 3$ and every integer $m \ge 2$, there is a constant $c>0$ such that
\[R(K_t,K_{m,m}) \geq c\left(\frac{m}{\log{m}}\right)^{(t+1)/2}.\]
\end{theorem}

To finish the proof, we now combine all the auxiliary results.

\begin{proof}[Proof of Theorem~\ref{thm-3UnifMaxDegLower}]
It suffices to consider $\mathcal{K}^{(3)}_4$, as ordered Ramsey numbers do not decrease by adding vertices. 
Let $n$ be a sufficiently large integer and let $m = \lceil n/4 \rceil$.
By Theorem~\ref{thm-ramCliqueBipartite}, we have $\overline{R}(\mathcal{K}_3,\mathcal{K}_{m,m}) \geq \Omega\left(\left(\frac{m}{\log{m}}\right)^2\right)$.
In particular, $2n < \overline{R}(\mathcal{K}_3,\mathcal{K}_{m,m})^\alpha$ for some fixed $\alpha < 1$ and a sufficiently large $n$.
Lemma~\ref{lem-3UnifMaxDegLowerPrecise} thus implies that 
\[\overline{R}(\mathcal{K}^{(3)}_4,\mathcal{K}^{(3)}_3(n)) \geq \left(\frac{m}{\log{m}}\right)^{\Omega(n)} \geq 2^{\Omega(n\log{n})},\]
which finishes the proof.
\end{proof}

\section{Proof of Proposition~\ref{prop-generalBound}}

For integers $\chi$ and $k$ with $\chi \geq k \geq 2$ and an ordered $k$-uniform hypergraph $\mathcal{H}$ on $n$ vertices with interval chromatic number $\chi$, we show that there is a constant $c$ such that
\[\overline{R}(\mathcal{H}) \leq 2^{R^{\chi(\chi-1)}  (c\chi  n)^{\chi-1}},\]
where $R=R(K^{(k)}_\chi)$.

Let $N=2^{ R^{\chi(\chi-1)}  (4\chi  n)^{\chi-1}}$ and consider a red-blue coloring of the hyperedges of $\mathcal{K}^{(k)}_N$ on $[N]$.
Since the interval chromatic number of $\mathcal{H}$ is $\chi$, we have $\mathcal{H} \subseteq \mathcal{K}^{(k)}_\chi(n)$ and thus it suffices to show that there is a monochromatic copy of $\mathcal{K}^{(k)}_\chi(n)$ in our coloring.

Since $R=R(K^{(k)}_\chi)$, every ordered subhypergraph of $\mathcal{K}^{(k)}_N$ induced by an $R$-tuple of vertices contains a monochromatic copy of $\mathcal{K}^{(k)}_\chi$ in our coloring.
Each such copy is contained in $\binom{N-\chi}{R-\chi}$ $R$-tuples of vertices of $\mathcal{K}^{(k)}_N$ and thus there are at least
\[\frac{\binom{N}{R}}{\binom{N-\chi}{R-\chi}} = \binom{R}{\chi}^{-1}\binom{N}{\chi}\]
monochromatic copies of $\mathcal{K}^{(k)}_\chi$.
Without loss of generality, we can assume that at least half of these copies has all hyperedges red.

Let $\mathcal{G}$ be the ordered $\chi$-uniform hypergraph with the vertex set $[N]$, where a $\chi$-tuple of vertices forms a hyperedge if and only if it induces a red copy of $\mathcal{K}^{(k)}_\chi$ in our coloring.
We know that $\mathcal{G}$ has at least $\frac{1}{2}\binom{R}{\chi}^{-1}\binom{N}{\chi} \geq \varepsilon \frac{N^\chi}{\chi!}$ hyperedges, where $\varepsilon = \frac{1}{4}\binom{R}{\chi}^{-1}$.
It is well-known (see~\cite{erdos64,niki09}, for example) that every $\chi$-uniform hypergraph on $N$ vertices with at least $\varepsilon \frac{N^\chi}{\chi!}$ hyperedges, where $(\ln{N})^{-1/(\chi-1)} \leq \varepsilon \leq \chi^{-3}$, contains a copy of $K^{(\chi)}_\chi(\lfloor \varepsilon(\ln{N})^{1/(\chi-1)} \rfloor)$.
Since \[\lfloor\varepsilon(\ln{N})^{1/(\chi-1)} \rfloor = \left\lfloor\frac{4R^\chi \chi n}{4\binom{R}{\chi}}\right\rfloor \geq \chi n,\]
we have a copy of $K^{(\chi)}_\chi(\chi n)$ in $\mathcal{G}$ in some ordering.
Note that every ordering of $K^{(\chi)}_\chi(\chi n)$ contains a copy of $\mathcal{K}^{(\chi)}_\chi(n)$ and thus $\mathcal{G}$ also contains a copy of $\mathcal{K}^{(\chi)}_\chi(n)$.
By the definition of $\mathcal{G}$, the vertices of this copy induce a red copy of $\mathcal{K}^{(k)}_\chi(n)$ in the coloring of~$\mathcal{K}^{(k)}_N$, which finishes the proof.
\qed 

\section{Open problems}
\label{sec-openProblems}

There is a plenty of open questions about ordered Ramsey numbers for ordered hypergraphs, as this area is still vastly unexplored.
Here we offer some of the open problems that we considered during our study.

We proved estimates on the ordered Ramsey numbers of ordered $3$-uniform hypergraphs with bounded maximum degree and with interval chromatic number $3$.
However, our bounds are not tight.
Recall that it follows from Theorem~\ref{thm-3UnifMaxDegIntChr} and Theorem~\ref{thm-3UnifMaxDegLower} that there are positive constants $c_1$, $c_2$, and $\varepsilon$, all depending on $d$, such that every ordered $3$-uniform hypergraph $\mathcal{H}$ on $n$ vertices with maximum degree $d$ satisfies
\[\overline{R}(\mathcal{H},\mathcal{K}^{(3)}_3(n)) \leq 2^{c_1n^{2-\varepsilon}},\]
while there is a fixed $\mathcal{G}$ such that 
\[\overline{R}(\mathcal{G},\mathcal{K}^{(3)}_3(n)) \geq 2^{c_2n\log{n}}.\]
Thus although the exponents in the bounds are reasonably close in this non-diagonal case, there is still a gap between them and it would be interesting to close it.

\begin{problem}
Let $d$ be a fixed positive integer and let $\mathcal{H}$ be an ordered $3$-uniform hypergraph on $n$ vertices with maximum degree $d$.
Close the gap between the lower and upper bounds on $\overline{R}(\mathcal{H},\mathcal{K}^{(3)}_3(n))$.
\end{problem}

Another interesting problem is to extend the upper bound from Corollary~\ref{cor-3UnifMaxDegIntChr} to ordered $3$-uniform hypergraphs with bounded maximum degree and fixed interval chromatic number that is larger than $3$.
We recall that their ordered Ramsey numbers are bounded from above by $2^{O(n^2)}$ by Corollary~\ref{cor-3UnifBoundedIntChr}.
In particular, one can ask if the exponent is still subquadratic in the number of vertices.

\begin{problem}
Let $d$ and $\chi$ be fixed positive integers.
Is there an $\varepsilon = \varepsilon(d,\chi)>0$ such that, for every ordered $3$-uniform hypergraph $\mathcal{H}$ on $n$ vertices with maximum degree $d$ and with interval chromatic number $\chi$, we have
\[\overline{R}(\mathcal{H}) \leq 2^{O(n^{2-\varepsilon})}?\]
\end{problem}

As we discussed in Subsection~\ref{subsec-previousResults}, the monotone hyperpaths are examples of ordered $3$-uniform hypergraphs with maximum degree $3$ and with an arbitrarily large interval chromatic number such that their ordered Ramsey numbers grow exponentially.
To our knowledge, this is the best lower bound on ordered Ramsey numbers of ordered $3$-uniform hypergraphs with bounded maximum degrees.
Can we obtain better lower bounds?

\begin{problem}
For a fixed $d$, is there an example of a family $\{\mathcal{H}_n\}$ of ordered $3$-uniform hypergraphs $\mathcal{H}_n$ on $n$ vertices with maximum degree $d$ such that the numbers $\overline{R}(\mathcal{H}_n)$ grow superexponentially in $n$?
\end{problem}

We note that even for $\mathcal{K}_n^{(3)}$ the best known lower bound on $\overline{R}(\mathcal{K}_n^{(3)})$ is of order $2^{\Omega(n^2)}$; see~\eqref{eq-hypergraphRamsey}.

The growth rate of Ramsey numbers of unordered $3$-uniform hypergraphs with bounded maximum degree is only linear in the number of vertices; see~\cite{cfs09}.
The situation is completely different for ordered $3$-uniform hypergraphs with bounded maximum degree, as we can, for example, see from the bounds for monotone hyperpaths.
In general, we are not aware of any nontrivial upper bounds on these ordered Ramsey numbers. 

\begin{problem}
What is the upper bound on ordered Ramsey numbers of ordered $3$-uniform hypergraphs with bounded maximum degree?
\end{problem}

Finally, we proved all our upper bounds under the assumption that the maximum degree is bounded.
However, the corresponding bounds for ordered graphs such as~\eqref{eq-degreeIntChrNumber} hold for bounded degeneracy, which is a less restrictive assumption.
Thus one can also try to strengthen our results, in particular Theorem~\ref{thm-3UnifMaxDegIntChr} and Corollary~\ref{cor-3UnifMaxDegIntChr}, to ordered hypergraphs with bounded degeneracy instead of the maximum degree.

\bibliography{bibliography}	

\begin{thebibliography}{10}

\bibitem{balko19}
Martin Balko.
\newblock Ramsey numbers and monotone colorings.
\newblock {\em J. Combin. Theory Ser. A}, 163:34--58, 2019.

\bibitem{bckk13}
Martin Balko, Josef Cibulka, Karel Kr{\'a}l, and Jan Kyn{\v c}l.
\newblock Ramsey numbers of ordered graphs.
\newblock {\em Electron. J. Combin.}, 27:P1.16, 2020.

\bibitem{bjv16}
Martin Balko, V{\' i}t Jel{\' i}nek, and Pavel Valtr.
\newblock On ordered {R}amsey numbers of bounded-degree graphs.
\newblock {\em J. Combin. Theory Ser. B}, 134:179--202, 2019.

\bibitem{balko2021ordered}
Martin Balko and M{\'a}t{\'e} Vizer.
\newblock On ordered ramsey numbers of tripartite 3-uniform hypergraphs.
\newblock In {\em Extended Abstracts EuroComb 2021}, pages 142--147. Springer,
  2021.

\bibitem{crst83}
V{\' a}clav Chv\'{a}tal, Vojt{\v e}ch R\"{o}dl, Endre Szemer\'{e}di, and
  William~T. Trotter, Jr.
\newblock The {R}amsey number of a graph with bounded maximum degree.
\newblock {\em J. Combin. Theory Ser. B}, 34(3):239--243, 1983.

\bibitem{clfs17}
David Conlon, Jacob Fox, Choongbum Lee, and Benny Sudakov.
\newblock Ordered {R}amsey numbers.
\newblock {\em J. Combin. Theory Ser. B}, 122:353--383, 2017.

\bibitem{cfs09}
David Conlon, Jacob Fox, and Benny Sudakov.
\newblock Ramsey numbers of sparse hypergraphs.
\newblock {\em Random Structures Algorithms}, 35(1):1--14, 2009.

\bibitem{cfs10}
David Conlon, Jacob Fox, and Benny Sudakov.
\newblock Hypergraph {R}amsey numbers.
\newblock {\em J. Amer. Math. Soc.}, 23(1):247--266, 2010.

\bibitem{cfs11}
David Conlon, Jacob Fox, and Benny Sudakov.
\newblock Large almost monochromatic subsets in hypergraphs.
\newblock {\em Israel J. Math.}, 181:423--432, 2011.

\bibitem{cfs12}
David Conlon, Jacob Fox, and Benny Sudakov.
\newblock Erd{\H o}s--{H}ajnal-type theorems in hypergraphs.
\newblock {\em J. Combin. Theory Ser. B}, 102(5):1142--1154, 2012.

\bibitem{cfsSurvey}
David Conlon, Jacob Fox, and Benny Sudakov.
\newblock Recent developments in graph {R}amsey theory.
\newblock In {\em Surveys in combinatorics 2015}, volume 424 of {\em London
  Math. Soc. Lecture Note Ser.}, pages 49--118. Cambridge Univ. Press,
  Cambridge, 2015.

\bibitem{cnko08}
Oliver Cooley, Nikolaos Fountoulakis, Daniela K\"{u}hn, and Deryk Osthus.
\newblock 3-uniform hypergraphs of bounded degree have linear {R}amsey numbers.
\newblock {\em J. Combin. Theory Ser. B}, 98(3):484--505, 2008.

\bibitem{cnko09}
Oliver Cooley, Nikolaos Fountoulakis, Daniela K\"{u}hn, and Deryk Osthus.
\newblock Embeddings and {R}amsey numbers of sparse {$k$}-uniform hypergraphs.
\newblock {\em Combinatorica}, 29(3):263--297, 2009.

\bibitem{erdos47}
Paul Erd\H{o}s.
\newblock Some remarks on the theory of graphs.
\newblock {\em Bull. Amer. Math. Soc.}, 53:292--294, 1947.

\bibitem{erdos64}
Paul Erd\H{o}s.
\newblock On extremal problems of graphs and generalized graphs.
\newblock {\em Israel J. Math.}, 2:183--190, 1964.

\bibitem{erdosRado65}
Paul Erd\H{o}s, Andr{\' a}s Hajnal, and Richard Rado.
\newblock Partition relations for cardinal numbers.
\newblock {\em Acta Math. Acad. Sci. Hungar.}, 16:93--196, 1965.

\bibitem{erdosSzekeres35}
Paul Erd{\H o}s and George Szekeres.
\newblock A combinatorial problem in geometry.
\newblock {\em Compositio Math.}, 2:463--470, 1935.

\bibitem{foxhe19}
Jacob Fox and Xiaoyu He.
\newblock Independent sets in hypergraphs with a forbidden link.
\newblock {\em arXiv preprint arXiv:1909.05988}, 2019.

\bibitem{fpss12}
Jacob Fox, J\'{a}nos Pach, Benny Sudakov, and Andrew Suk.
\newblock Erd{\H o}s-{S}zekeres-type theorems for monotone paths and convex
  bodies.
\newblock {\em Proc. Lond. Math. Soc. (3)}, 105(5):953--982, 2012.

\bibitem{ishingami07}
Yoshiyasu Ishigami.
\newblock Linear {R}amsey numbers for bounded-degree hypergrahps.
\newblock {\em Electronic Notes in Discrete Mathematics}, 29:47 -- 51, 2007.
\newblock European Conference on Combinatorics, Graph Theory and Applications.

\bibitem{liZan03}
Yusheng Li and Wenan Zang.
\newblock Ramsey numbers involving large dense graphs and bipartite {T}ur\'{a}n
  numbers.
\newblock {\em J. Combin. Theory Ser. B}, 87(2):280--288, 2003.

\bibitem{moshShap14}
Guy Moshkovitz and Asaf Shapira.
\newblock Ramsey theory, integer partitions and a new proof of the {E}rd{\H
  o}s-{S}zekeres theorem.
\newblock {\em Adv. Math.}, 262:1107--1129, 2014.

\bibitem{nsrs08}
Brendan Nagle, Sayaka Olsen, Vojt\v{e}ch R\"{o}dl, and Mathias Schacht.
\newblock On the {R}amsey number of sparse 3-graphs.
\newblock {\em Graphs Combin.}, 24(3):205--228, 2008.

\bibitem{niki09}
Vladimir Nikiforov.
\newblock Complete {$r$}-partite subgraphs of dense {$r$}-graphs.
\newblock {\em Discrete Math.}, 309(13):4326--4331, 2009.

\bibitem{pachTardos06}
J\'{a}nos Pach and G\'{a}bor Tardos.
\newblock Forbidden paths and cycles in ordered graphs and matrices.
\newblock {\em Israel J. Math.}, 155:359--380, 2006.

\end{thebibliography}
\bibliographystyle{plain}

\end{document}